\documentclass[12pt]{amsart}

\usepackage{amsmath,amssymb,amsthm,mathtools}
\usepackage[a4paper,margin=1.05in]{geometry}
\usepackage{enumitem}
\usepackage{hyperref}
\usepackage[nameinlink,capitalize]{cleveref}
\usepackage{microtype}

\hypersetup{colorlinks=true,linkcolor=blue,citecolor=blue,urlcolor=blue}
\newtheorem{theorem}{Theorem}[section]
\newtheorem{proposition}[theorem]{Proposition}
\newtheorem{lemma}[theorem]{Lemma}
\newtheorem{corollary}[theorem]{Corollary}

\newtheorem{remark}[theorem]{Remark}
\theoremstyle{definition}

\crefname{theorem}{Theorem}{Theorems}
\crefname{lemma}{Lemma}{Lemmas}
\crefname{corollary}{Corollary}{Corollaries}
\crefname{proposition}{Proposition}{Propositions}
\crefname{remark}{Remark}{Remarks}
\crefname{claim}{Claim}{Claims}

\Crefname{theorem}{Theorem}{Theorems}
\Crefname{lemma}{Lemma}{Lemmas}
\Crefname{corollary}{Corollary}{Corollaries}
\Crefname{proposition}{Proposition}{Propositions}
\Crefname{remark}{Remark}{Remarks}
\Crefname{claim}{Claim}{Claims}

\newcommand{\C}{\mathbb C}
\newcommand{\N}{\mathbb N}
\newcommand{\Z}{\mathbb Z}
\newcommand{\D}{\mathbb D}
\newcommand{\B}{\mathbb B}
\newcommand{\BB}{\mathcal B}

\newcommand{\dV}{\,dV}

\newcommand{\less}{\lesssim}
\newcommand{\gtr}{\gtrsim}

\newcommand{\ip}[2]{\langle #1,#2\rangle}
\newcommand{\calB}{\mathcal B}

\newcommand{\abs}[1]{\left|#1\right|}
\newcommand{\norm}[1]{\left\|#1\right\|}

\title[$L^p$-boundedness of Berezin transforms on generalized Hartogs triangles]
{$L^p$-Boundedness of the Berezin Transform on Generalized Hartogs Triangles}

\author{Qian Fu}
\address{Teaching Experiment Platform, Beijing Normal University, Zhuhai 519087, P. R. China}
\email{qianfu@bnu.edu.cn}


\begin{document}

\begin{abstract}
Let $m,l\in\N$ be relatively prime and let
\[
\Omega_{m/l}^{n+1}=\bigl\{(z,w)\in\C^n\times\C:\ \norm{z}^{m}<\abs{w}^{l}<1\bigr\}
\]
be the rational generalized Hartogs triangle of exponent $m/l$ in $\C^{n+1}$. In this paper, we study the Berezin transform $\BB_{m/l,n}$ associated with the Bergman kernel of $\Omega_{m/l}^{n+1}$, and prove that $\BB_{m/l,n}$ is bounded on $L^p(\Omega_{m/l}^{n+1})$ if and only if $p>m+nl$. 
\end{abstract}

\maketitle

\section{Introduction}

The Berezin transform is obtained by averaging a function against the squared
modulus of the normalized Bergman kernel.  It is therefore a positive integral
operator whose mapping properties reflect the way in which the Bergman kernel
blows up near the boundary.  More precisely, if $\Omega\subset\C^N$ is a bounded
domain with Bergman kernel $B_\Omega$, then its Berezin transform is given by
\begin{equation}\label{eq:intro-berezin}
  \calB_\Omega f(\zeta)
  =\int_\Omega f(\omega)\frac{|B_\Omega(\zeta,\omega)|^2}{B_\Omega(\zeta,\zeta)}\dV(\omega).
\end{equation}
The normalization in the denominator makes \eqref{eq:intro-berezin} bounded on
$L^\infty(\Omega)$ for every bounded domain.  For finite $p$, however, the
answer can change drastically from one domain to another, especially on
nonsmooth domains where the Bergman kernel has anisotropic singularities.

On smooth bounded symmetric domains and on classical homogeneous domains, the
Berezin transform has a long history in operator theory and quantization.  The
boundary behavior of the Berezin transform was used by Axler and Zheng
\cite{AxlerZheng1998} to characterize compactness phenomena on the disk, and
Engli\v{s}\cite{Englis1999} extended related ideas to bounded symmetric domains.
The $L^p$ theory on the disk and ball was studied by Dostani'c
\cite{Dostanic2008}, Markovi'c \cite{Markovic2015}, and others.  More recently,
there has been growing interest in nonsmooth pseudoconvex domains, where
Bergman-type integral operators often have restricted mapping ranges, see \cite{ChakrabartiShaw2013,ChakrabartiZeytuncu2016,Chen2017,ChristophersonKoenig2024,ChristophersonChen2025},  This is
particularly visible on Hartogs triangles and related Reinhardt domains; see, for
example, \cite{HuoWick2020,QinWuGuo2023,YangZou2024,Zou2024Note,Zou2026FR,FuDengCao2024}.  

The present paper concerns the rational generalized Hartogs triangles
\begin{equation}\label{eq:domain-intro}
   \Omega^{n+1}_{m/l}
   =\{(z,w)\in\C^n\times\C:\ \norm{z}^m<|w|^l<1\},
   \qquad (m,l)=1.
\end{equation}
They are special cases of the power-generalized domains
\[
   \Omega^{n+1}_{\gamma}
   =\{(z,w)\in\C^n\times\C:\ \norm{z}^{\gamma}<|w|<1\},
   \qquad \gamma>0,
\]
with $\gamma=m/l$. These domains are higher-dimensional generalizations of the classical Hartogs triangle \cite{EdholmMcNeal2016,EdholmMcNeal2017,GogusSahutoglu2023}.  In our previous work \cite{FuDeng2026}, we obtained the
Bergman kernel estimates and the sharp $L^p$ mapping range of the Bergman
projection on these domains.  When $m=1$ and $l=k$, the
domain \eqref{eq:domain-intro} becomes
\[
   H_k=\{(z,w)\in\C^n\times\C:\ \norm{z}<|w|^k<1\},
\]
the generalized Hartogs triangle studied by Zou \cite{Zou2025Berezin}.  Zou
proved that the Berezin transform on $H_k$ is bounded on $L^p(H_k)$ if and only
if $p>nk+1$.  His proof uses the explicit closed formula for the Bergman kernel
of $H_k$, which is available because $H_k$ is biholomorphic to
$\B_n\times\D^*$ by the holomorphic map $(z,w)\mapsto (z/w^k,w)$.

A complementary two-dimensional endpoint of the family was recently studied by
Wei \cite{Wei2025Fat}.  For the fat Hartogs triangle
\[
   \{(z_1,z_2)\in\C^2:\ |z_1|^k<|z_2|<1\},
\]
Wei proved the sharp statement that the Berezin transform is bounded precisely
for $p>k+1$.  This corresponds to the case $n=1$ and $l=1$ in
\eqref{eq:domain-intro}.

For the domains \eqref{eq:domain-intro} with $m>1$, the situation is different.
The map involving $w^{l/m}$ is not globally holomorphic, and the Bergman kernel
is no longer a single Forelli--Rudin fraction.  Fu--Deng \cite{FuDeng2026} obtained the Bergman
kernel by decomposing the Bergman space into finitely many orthogonal subspaces
according to congruence classes of the total degree.  This leads to a finite sum
of subkernels.  Their kernel estimates were used to prove the sharp
$L^p$-boundedness interval for the Bergman projection.  The same decomposition,
combined with a different normalization, is the starting point of our analysis of
the Berezin transform.

Our main result is the following sharp theorem.

\begin{theorem}\label{thm:main}
Let $B_{m/l,n}$ be the Bergman kernel of $\Omega^{n+1}_{m/l}$ and define
\begin{equation*}
   \calB_{m/l,n}f(z,w)
   =\int_{\Omega^{n+1}_{m/l}}
      \frac{|B_{m/l,n}((z,w),(\xi,\eta))|^2}{B_{m/l,n}((z,w),(z,w))}f(\xi,\eta)\dV(\xi,\eta).
\end{equation*}
If $1<p<\infty$, then
\[
   \calB_{m/l,n}:L^p(\Omega^{n+1}_{m/l})\to L^p(\Omega^{n+1}_{m/l})
\]
is bounded if and only if
\begin{equation*}
   p>m+nl.
\end{equation*}
\end{theorem}

\begin{remark}
	When $m=1$ and $l=k$, the domain becomes
	$\{(z,w):\norm{z}<|w|^k<1\}$, and the result reduces to Zou's theorem \cite[Theorem~1.2]{Zou2025Berezin} that the Berezin transform is $L^p$-bounded exactly for $p>nk+1$. In the case $n=1$ and $l=1$, it also recovers the sharp fat-Hartogs-triangle theorem of Wei \cite[Theorems~1.1 and~1.2]{Wei2025Fat}.
\end{remark}

\begin{corollary}\label{cor:unbounded}
For $p\ge1$, the Berezin transform $\calB_{m/l,n}$ is unbounded on
$L^p(\Omega^{n+1}_{m/l})$ if and only if
\[
   1\le p\le m+nl.
\]
\end{corollary}

We shall use the following standard notation throughout the paper. If $A$ and
$B$ are functions depending on several variables, we write $A\lesssim B$ to signify that there exists a
constant $K>0$, independent of the relevant variables, such that $A\leqslant K B$. The independence
of which variables will be clear in context. We also write $A\approx B$ to mean that $A\lesssim B\lesssim A$. If $x\in\mathbb R$, $\lfloor x\rfloor$ will denote the greatest integer not exceeding $x$. We write
$U\Subset\Omega$ to mean that $U$ is compactly contained in $\Omega$.

\section{Bergman kernels on rational generalized Hartogs triangles}

\subsection{The monomial expansion}

For $\alpha=(\alpha_1,\ldots,\alpha_n)\in\N^n$ and $\beta\in\Z$, the monomial
$z^\alpha w^\beta$ belongs to $A^2(\Omega^{n+1}_{m/l})$ if and only if
\begin{equation}\label{eq:monomial-condition}
   |\alpha|+\frac{m}{l}(\beta+1)>-n.
\end{equation}
Equivalently,
\begin{equation}\label{eq:integer-monomial-condition}
   l|\alpha|+m\beta\ge -m-nl+1.
\end{equation}
Indeed, by polar coordinates,
\begin{align*}
 \norm{z^\alpha w^\beta}_{L^2(\Omega^{n+1}_{m/l})}^2
 &=\int_{0<|w|<1}|w|^{2\beta}
     \int_{\norm{z}<|w|^{l/m}} |z^\alpha|^2\dV(z)\dV(w) \notag\\
 &= C_{n,\alpha}\int_0^1 r^{2\beta+\frac{2l}{m}(|\alpha|+n)+1}\,dr,
\end{align*}
which is finite exactly when \eqref{eq:monomial-condition} holds.
Consequently the Bergman kernel has the orthogonal monomial expansion
\begin{equation}\label{eq:full-kernel-expansion}
   B_{m/l,n}((z,w),(\xi,\eta))
   =\sum_{(\alpha,\beta)\in\Lambda_{m/l,n}}
     \frac{z^\alpha w^\beta\overline{\xi^\alpha\eta^\beta}}
     {\norm{z^\alpha w^\beta}_{L^2(\Omega^{n+1}_{m/l})}^2},
\end{equation}
where
\begin{equation*}
   \Lambda_{m/l,n}
   =\{(\alpha,\beta)\in\N^n\times\Z:\ l|\alpha|+m\beta\ge -m-nl+1\}.
\end{equation*}

We shall use the following finite congruence-class decomposition.  For
$j\in\{0,1,\ldots,m-1\}$, let $\mathcal H_j$ be the closed subspace of
$A^2(\Omega^{n+1}_{m/l})$ spanned by the monomials $z^\alpha w^\beta$ satisfying
$|\alpha|\equiv j\pmod m$.  Then
\begin{equation*}
   A^2(\Omega^{n+1}_{m/l})=\mathcal H_0\oplus\mathcal H_1\oplus\cdots\oplus\mathcal H_{m-1}.
\end{equation*}
Let $K_j$ denote the reproducing kernel of $\mathcal H_j$.  Then
\begin{equation}\label{eq:kernel-decomp}
   B_{m/l,n}((z,w),(\xi,\eta))=\sum_{j=0}^{m-1} K_j((z,w),(\xi,\eta)).
\end{equation}

For each $j$ set
\begin{equation*}
   E_j=\left\lfloor\frac{(j+n)l-1}{m}\right\rfloor,
   \qquad
   \kappa_j=E_j+1-\frac{lj}{m}.
\end{equation*}
Then
\begin{equation*}
   \kappa_j\le \frac{nl+m-1}{m},
\end{equation*}
and equality holds for the unique $j=j_0$ satisfying
\begin{equation*}
   l(j_0+n)\equiv1\pmod m.
\end{equation*}

\subsection{Bergman kernel estimates}

We shall use the following estimates for the Bergman kernel and for the normalized
Berezin kernel.  

\begin{lemma}\label{lem:bergman-kernel-estimates}
Let $B_{m/l,n}$ be the Bergman kernel of $\Omega^{n+1}_{m/l}$.  For all
$(z,w),(\xi,\eta)\in\Omega^{n+1}_{m/l}$,
\begin{equation}\label{eq:full-kernel-upper}
 |B_{m/l,n}((z,w),(\xi,\eta))|
 \less
 \frac{|w\overline\eta|^{-\frac{nl+m-1}{m}}}
 {|1-w\overline\eta|^2
 \left|1-\frac{\ip{z}{\xi}^m}{(w\overline\eta)^l}\right|^{n+1}}.
\end{equation}
Moreover, the Berezin kernel satisfies
\begin{equation}\label{eq:berezin-kernel-upper}
 \frac{|B_{m/l,n}((z,w),(\xi,\eta))|^2}{B_{m/l,n}((z,w),(z,w))}
 \less
 \frac{\left(1-\frac{\norm{z}^{2m}}{|w|^{2l}}\right)^{n+1}(1-|w|^2)^2
 |\eta|^{-2\frac{nl+m-1}{m}}}
 {|1-w\overline\eta|^4
 \left|1-\frac{\ip{z}{\xi}^m}{(w\overline\eta)^l}\right|^{2(n+1)}}.
\end{equation}
\end{lemma}

\begin{proof}
The estimate \eqref{eq:full-kernel-upper} is exactly the full Bergman kernel
estimate in \cite[Corollary~2.4]{FuDeng2026}, after writing
\[
|(w\overline\eta)^l-\ip{z}{\xi}^m|^{n+1}
=|w\overline\eta|^{l(n+1)}
\left|1-\frac{\ip{z}{\xi}^m}{(w\overline\eta)^l}\right|^{n+1}.
\]
It remains to prove \eqref{eq:berezin-kernel-upper}.  Let
$B_{m/l,n}=\sum_{j=0}^{m-1}K_j$ be the Fu--Deng decomposition into sub-Bergman
kernels \cite{FuDeng2026}.  From formula (2.16) in the proof of \cite[Theorem~2.3]{FuDeng2026},
after polarization and factoring out $|w\overline\eta|^{l(n+1)}$, we have
\begin{equation}\label{eq:subkernel-estimate-for-proof}
	|K_j((z,w),(\xi,\eta))|
	\less
	\frac{|\ip{z}{\xi}|^j |w\overline\eta|^{-E_j-1}}
	{|1-w\overline\eta|^2
		\left|1-\frac{\ip{z}{\xi}^m}{(w\overline\eta)^l}\right|^{n+1}}.
\end{equation}
On the diagonal we also need a lower bound.  Put
\[
a=\norm{z}^2,\qquad b=|w|^2,\qquad u=ab^{-l/m}.
\]
Then $0<u<1$ on $\Omega^{n+1}_{m/l}$.  By formula (2.15) in the proof of
Fu--Deng's subkernel estimate \cite[Theorem~2.3]{FuDeng2026}, the diagonal
value of the $j$-th subkernel is written as
\[
K_j((z,w),(z,w))
=
\frac{l}{m\pi^{n+1}}\,g_j(b)\,
\frac{b^{\frac{lj}{m}-1-E_j}}{(1-b)^2}
\frac{d^n}{du^n}\left(\frac{u^{j+n}}{1-u^m}\right),
\]
where
\[
g_j(b)=j+n-\frac{m}{l}E_j+
\left(\frac{m}{l}+\frac{m}{l}E_j-j-n\right)b.
\]
In particular,
\[
g_j(0)=j+n-\frac{m}{l}E_j>0,
\qquad
g_j(1)=\frac{m}{l}>0.
\]
Hence $g_j(b)\gtrsim1$ for $0<b<1$.

It remains to estimate the derivative term from below.  Since
\[
        \frac{u^{j+n}}{1-u^m}
        =\sum_{k=0}^{\infty}u^{j+n+mk},\qquad 0<u<1,
\]
we have
\[
\frac{d^n}{du^n}\left(\frac{u^{j+n}}{1-u^m}\right)
 =
 \sum_{k=0}^{\infty}
 (j+n+mk)(j+n+mk-1)\cdots(j+mk+1)u^{j+mk}.
\]
All coefficients in this series are positive, and
\[
 (j+n+mk)(j+n+mk-1)\cdots(j+mk+1)\gtrsim(k+1)^n.
\]
Consequently,
\[
\frac{d^n}{du^n}\left(\frac{u^{j+n}}{1-u^m}\right)
 \gtrsim
 u^j\sum_{k=0}^{\infty}(k+1)^n(u^m)^k
 \approx
 \frac{u^j}{(1-u^m)^{n+1}}.
\]
Substituting this estimate into the diagonal formula gives
\[
K_j((z,w),(z,w))
 \gtrsim
 \frac{b^{lj/m-1-E_j}}{(1-b)^2}
 \frac{u^j}{(1-u^m)^{n+1}}.
\]
Since $u=ab^{-l/m}$, we get
\[
        b^{lj/m-1-E_j}u^j
        =a^j b^{-1-E_j}.
\]
Thus
\[
K_j((z,w),(z,w))
 \gtrsim
 \frac{a^j b^{-1-E_j}}{(1-b)^2(1-u^m)^{n+1}}.
\]
Returning to $a=\norm{z}^2$, $b=|w|^2$, and
$u^m=\norm{z}^{2m}/|w|^{2l}$, we obtain
\begin{equation}\label{eq:subkernel-diagonal-proof}
 K_j((z,w),(z,w))
 \gtr
 \frac{\norm{z}^{2j}|w|^{-2(E_j+1)}}
 {(1-|w|^2)^2\left(1-\frac{\norm{z}^{2m}}{|w|^{2l}}\right)^{n+1}},
\end{equation}
Since $B_{m/l,n}((z,w),(z,w))\ge K_j((z,w),(z,w))$, estimates
\eqref{eq:subkernel-estimate-for-proof} and \eqref{eq:subkernel-diagonal-proof}
give
\begin{align*}
 \frac{|K_j((z,w),(\xi,\eta))|^2}{B_{m/l,n}((z,w),(z,w))}
 &\less
 \frac{\left(1-\frac{\norm{z}^{2m}}{|w|^{2l}}\right)^{n+1}(1-|w|^2)^2}
 {|1-w\overline\eta|^4
 \left|1-\frac{\ip{z}{\xi}^m}{(w\overline\eta)^l}\right|^{2(n+1)}}       \\
 &\qquad\times
 \frac{|\ip{z}{\xi}|^{2j}}{\norm{z}^{2j}} |\eta|^{-2(E_j+1)}.
\end{align*}
If $z\ne0$, then
\[
   \frac{|\ip{z}{\xi}|^{2j}}{\norm{z}^{2j}}
   \le \norm{\xi}^{2j}
   \le |\eta|^{2lj/m},
\]
and the same conclusion is immediate for $z=0$ by continuity.  Hence
\begin{equation*}
 \frac{|K_j((z,w),(\xi,\eta))|^2}{B_{m/l,n}((z,w),(z,w))}
 \less
 \frac{\left(1-\frac{\norm{z}^{2m}}{|w|^{2l}}\right)^{n+1}(1-|w|^2)^2
 |\eta|^{-2\kappa_j}}
 {|1-w\overline\eta|^4
 \left|1-\frac{\ip{z}{\xi}^m}{(w\overline\eta)^l}\right|^{2(n+1)}},
\end{equation*}
where $\kappa_j=E_j+1-\frac{lj}{m}$. Since $\kappa_j\le (nl+m-1)/m$ and $0<|\eta|<1$, we may replace
$|\eta|^{-2\kappa_j}$ by $|\eta|^{-2(nl+m-1)/m}$.  Finally, by the decomposition \(B_{m/l,n}=\sum_{j=0}^{m-1}K_j\) and the Cauchy--Schwarz inequality for finite sums, we have
\[
 |B_{m/l,n}((z,w),(\xi,\eta))|^2
 \le m\sum_{j=0}^{m-1}|K_j((z,w),(\xi,\eta))|^2,
\]
and \eqref{eq:berezin-kernel-upper} follows.
\end{proof}

\section{Integral estimates and Schur functions}

\begin{lemma}\label{lem:ball-estimate}
Let $\alpha>-1$ and $\tau>0$.  Then for $a\in\B_n$,
\begin{equation}\label{eq:ball-estimate}
   I_{\alpha,\tau}(a)
   =\int_{\B_n}\frac{(1-\norm{u}^{2m})^\alpha}
   {|1-\ip{a}{u}^m|^{n+1+\alpha+\tau}}\dV(u)
   \less
   \bigl(1-\norm{a}^{2m}\bigr)^{-\tau}.
\end{equation}
In particular, if $\Gamma>n+1+\alpha$, then
\begin{equation}\label{eq:ball-estimate-gamma}
   \int_{\B_n}\frac{(1-\norm{u}^{2m})^\alpha}
   {|1-\ip{a}{u}^m|^{\Gamma}}\dV(u)
   \less
   \bigl(1-\norm{a}^{2m}\bigr)^{\alpha+n+1-\Gamma}.
\end{equation}
\end{lemma}
\begin{proof}
By unitary invariance of the unit ball and Lebesgue
measure, we may assume that
\[
        a=\norm{a}e_1,\qquad e_1=(1,0,\ldots,0).
\]
Put $r=\norm{a}$ and
\[
        \lambda=\frac{n+1+\alpha+\tau}{2}>0.
\]
For $|\zeta|<1$, the binomial expansion gives
\[
        \frac{1}{(1-\zeta)^\lambda}
        =
        \sum_{k=0}^{\infty}
        \frac{\Gamma(k+\lambda)}{\Gamma(\lambda)\Gamma(k+1)}
        \zeta^k .
\]
Since the mixed terms vanish after integration over the angular variable, we obtain
\begin{align*}
 I_{\alpha,\tau}(a)
 &=
 \int_{\B_n}
 \frac{(1-\norm{u}^{2m})^\alpha}
 {|1-(r u_1)^m|^{2\lambda}}\dV(u) \\
 &=
 \sum_{k=0}^{\infty}
 \left(
 \frac{\Gamma(k+\lambda)}{\Gamma(\lambda)\Gamma(k+1)}
 \right)^2
 r^{2mk}
 \int_{\B_n}(1-\norm{u}^{2m})^\alpha |u_1|^{2mk}\dV(u).
\end{align*}
It remains to estimate the last integral.  Using polar coordinates
$u=\rho\zeta$, $\zeta\in\partial\B_n$, and the standard sphere integral (see, for example, \cite{Rudin1980,Zhu2005}),
\[
        \int_{\partial\B_n}|\zeta_1|^{2mk}\,d\sigma(\zeta)
        =
        C_n\frac{\Gamma(mk+1)}{\Gamma(n+mk)},
\]
we get
\begin{align*}
&\int_{\B_n}(1-\norm{u}^{2m})^\alpha |u_1|^{2mk}\dV(u)  \\
&\quad =
 C_n
 \frac{\Gamma(mk+1)}{\Gamma(n+mk)}
 \int_0^1(1-\rho^{2m})^\alpha \rho^{2n+2mk-1}\,d\rho \\
&\quad =
 C_{m,n}
 \frac{\Gamma(mk+1)}{\Gamma(n+mk)}
 \int_0^1(1-t)^\alpha t^{k+\frac{n}{m}-1}\,dt \\
&\quad =
 C_{\alpha,m,n}
 \frac{\Gamma(mk+1)}{\Gamma(n+mk)}
 \frac{\Gamma(k+\frac{n}{m})}
 {\Gamma(k+\frac{n}{m}+\alpha+1)} .
\end{align*}
By Stirling's formula,
\[
        \frac{\Gamma(k+\lambda)}{\Gamma(k+1)}\approx (k+1)^{\lambda-1},
        \qquad
        \frac{\Gamma(k+\frac{n}{m})}
        {\Gamma(k+\frac{n}{m}+\alpha+1)}
        \approx (k+1)^{-\alpha-1},
\]
and
\[
        \frac{\Gamma(mk+1)}{\Gamma(n+mk)}\approx (k+1)^{1-n}.
\]
Hence the coefficient of $r^{2mk}$ in the preceding series is bounded by a
constant multiple of
\[
        (k+1)^{2\lambda-2}(k+1)^{-\alpha-1}(k+1)^{1-n}
        =
        (k+1)^{\tau-1}.
\]
Consequently,
\[
        I_{\alpha,\tau}(a)
        \less
        \sum_{k=0}^{\infty}(k+1)^{\tau-1}r^{2mk}.
\]
Since $\tau>0$, the standard power-series estimate gives
\[
        \sum_{k=0}^{\infty}(k+1)^{\tau-1}x^k
        \less
        (1-x)^{-\tau},
        \qquad 0\le x<1.
\]
Taking $x=r^{2m}=\norm{a}^{2m}$, we obtain
\[
        I_{\alpha,\tau}(a)
        \less
        (1-\norm{a}^{2m})^{-\tau},
\]
which proves \eqref{eq:ball-estimate}.

Finally, if $\Gamma>n+1+\alpha$, then taking
$\tau=\Gamma-n-1-\alpha>0$ in \eqref{eq:ball-estimate} gives
\[
   \int_{\B_n}\frac{(1-\norm{u}^{2m})^\alpha}
   {|1-\ip{a}{u}^m|^{\Gamma}}\dV(u)
   \less
   \bigl(1-\norm{a}^{2m}\bigr)^{\alpha+n+1-\Gamma}.
\]
This proves \eqref{eq:ball-estimate-gamma}.
\end{proof}

The second estimate is the disk estimate used in Schur's test.

\begin{lemma}\cite[Lemma~2.2]{Zou2025Berezin}\label{lem:disk-estimate}
Let $\alpha>-1$, $\beta>-2$, and $\Gamma>\alpha+2$.  Then
\begin{equation*}
   \int_{\D^*}\frac{(1-|\eta|^2)^\alpha |\eta|^\beta}
   {|1-w\overline\eta|^{\Gamma}}\dV(\eta)
   \less (1-|w|^2)^{\alpha+2-\Gamma},
   \qquad w\in\D^*.
\end{equation*}
The implicit constant is independent of $w$.
\end{lemma}

We shall use Schur's test in the following form.

\begin{lemma}[Schur test]\label{lem:schur}
Let $(X,\mu)$ be a measure space, $1<p<\infty$, and $q=p/(p-1)$.  Let $T$ be a
positive integral operator
\[
   Tf(x)=\int_X K(x,y)f(y)\,d\mu(y),
   \qquad K(x,y)\ge0.
\]
If there exist a positive function $h$ and a constant $C>0$ such that
\[
   \int_X K(x,y)h(y)^q\,d\mu(y)\le C h(x)^q
\]
and
\[
   \int_X K(x,y)h(x)^p\,d\mu(x)\le C h(y)^p,
\]
then $T$ is bounded on $L^p(X,d\mu)$.
\end{lemma}
This is the standard Schur test; see, for example,
\cite[Theorem~2.5]{Zou2025Berezin}.

\section{Sufficiency}

We prove that $p>m+nl$ implies $L^p$-boundedness of $\calB_{m/l,n}$.  By
\eqref{eq:berezin-kernel-upper}, it is enough to prove Schur estimates for the
positive kernel
\begin{equation}\label{eq:global-positive-kernel}
 K^+((z,w),(\xi,\eta))
 =
 \frac{\left(1-\frac{\norm{z}^{2m}}{|w|^{2l}}\right)^{n+1}(1-|w|^2)^2
 |\eta|^{-2\frac{nl+m-1}{m}}}
 {|1-w\overline\eta|^4
 \left|1-\frac{\ip{z}{\xi}^m}{(w\overline\eta)^l}\right|^{2(n+1)}}.
\end{equation}

\begin{lemma}\label{lem:global-intervals}
Assume $p>m+nl$ and let $q=p/(p-1)$.  Then the interval
\begin{equation}\label{eq:global-s-interval}
  \left(-\frac{2}{mq},0\right]
  \cap
  \left(-\frac{2+\frac{2nl}{m}}{p},-\frac{2(nl+m-1)}{mp}\right)
\end{equation}
is nonempty.  Moreover, the intervals
\begin{equation}\label{eq:global-gamma1-interval}
  \left(-\frac1q,\frac{n+1}{q}\right)
  \cap
  \left(-\frac{n+2}{p},0\right)
\end{equation}
and
\begin{equation}\label{eq:global-gamma2-interval}
  \left(-\frac1q,\frac2q\right)
  \cap
  \left(-\frac3p,0\right)
\end{equation}
are nonempty.
\end{lemma}

\begin{proof}
The intervals \eqref{eq:global-gamma1-interval} and
\eqref{eq:global-gamma2-interval} are nonempty because $p,q>1$.  For
\eqref{eq:global-s-interval}, the second lower endpoint is strictly smaller than
its upper endpoint since
\[
   -\frac{2+\frac{2nl}{m}}{p}<-\frac{2(nl+m-1)}{mp}.
\]
Thus the only nontrivial condition is
\[
   -\frac{2}{mq}<-\frac{2(nl+m-1)}{mp}.
\]
Since $q=p/(p-1)$, this is equivalent to
\[
   p-1>nl+m-1,
\]
that is, $p>m+nl$.
\end{proof}

\begin{proposition}\label{prop:sufficiency}
If $1<p<\infty$ and $p>m+nl$, then $\calB_{m/l,n}$ is bounded on
$L^p(\Omega^{n+1}_{m/l})$.
\end{proposition}

\begin{proof}
Let $q=p/(p-1)$.  By Lemma~\ref{lem:global-intervals}, choose
\begin{align*}
\gamma_1&\in
\left(-\frac1q,\frac{n+1}{q}\right)
\cap
\left(-\frac{n+2}{p},0\right),\\
\gamma_2&\in
\left(-\frac1q,\frac2q\right)
\cap
\left(-\frac3p,0\right),\\
s&\in
\left(-\frac{2}{mq},0\right]
\cap
\left(-\frac{2+\frac{2nl}{m}}{p},-\frac{2(nl+m-1)}{mp}\right).
\end{align*}
Set
\begin{equation*}
   h(z,w)=\left(1-\frac{\norm{z}^{2m}}{|w|^{2l}}\right)^{\gamma_1}(1-|w|^2)^{\gamma_2}|w|^s.
\end{equation*}
We verify Schur's test for the positive kernel \eqref{eq:global-positive-kernel}.

First fix $(z,w)\in\Omega^{n+1}_{m/l}$ and estimate
\[
   \int_{\Omega^{n+1}_{m/l}} K^+((z,w),(\xi,\eta))h(\xi,\eta)^q\dV(\xi,\eta).
\]
Put
\[
    \xi=|\eta|^{l/m}u,
    \qquad u\in\B_n.
\]
Then $dV(\xi)=|\eta|^{\frac{2nl}{m}}dV(u)$ and
\[
   1-\frac{\norm{\xi}^{2m}}{|\eta|^{2l}}=1-\norm{u}^{2m}.
\]
It follows that
\begin{align*}
&\int_{\Omega^{n+1}_{m/l}} K^+((z,w),(\xi,\eta))h(\xi,\eta)^q\dV(\xi,\eta) \\
&= 
\left(1-\frac{\norm{z}^{2m}}{|w|^{2l}}\right)^{n+1}(1-|w|^2)^2 \\
&\quad \times
\int_{\D^*}\int_{\B_n}
\frac{(1-\norm{u}^{2m})^{\gamma_1 q}(1-|\eta|^2)^{\gamma_2 q}
|\eta|^{\frac{2nl}{m}-2\frac{nl+m-1}{m}+sq}}
{|1-w\overline\eta|^4
\left|1-\frac{\ip{z}{|\eta|^{l/m}u}^m}{(w\overline\eta)^l}\right|^{2(n+1)}}
\dV(u)\dV(\eta).
\end{align*}
The ball part is bounded by Lemma~\ref{lem:ball-estimate}:
\begin{align*}
 &\int_{\B_n}\frac{(1-\norm{u}^{2m})^{\gamma_1 q}}
 {\left|1-\frac{\ip{z}{|\eta|^{l/m}u}^m}{(w\overline\eta)^l}\right|^{2(n+1)}}\dV(u) \\
 &\hspace{1cm}\less
 \left(1-\frac{\norm{z}^{2m}}{|w|^{2l}}\right)^{\gamma_1q-(n+1)},
\end{align*}
because $\gamma_1q>-1$ and $\gamma_1q-(n+1)<0$.  The disk part is
\[
  \int_{\D^*}
  \frac{(1-|\eta|^2)^{\gamma_2q}|\eta|^{\frac{2nl}{m}-2\frac{nl+m-1}{m}+sq}}
  {|1-w\overline\eta|^4}\dV(\eta).
\]
The choice of $s$ gives
\[
   \frac{2nl}{m}-2\frac{nl+m-1}{m}+sq>-2,
\]
and the choice of $\gamma_2$ gives $\gamma_2q>-1$ and $\gamma_2q-2<0$.  Hence
Lemma~\ref{lem:disk-estimate} yields
\[
  \int_{\D^*}
  \frac{(1-|\eta|^2)^{\gamma_2q}|\eta|^{\frac{2nl}{m}-2\frac{nl+m-1}{m}+sq}}
  {|1-w\overline\eta|^4}\dV(\eta)
  \less (1-|w|^2)^{\gamma_2q-2}.
\]
Multiplying by the outside factor
$\left(1-\frac{\norm{z}^{2m}}{|w|^{2l}}\right)^{n+1}(1-|w|^2)^2$, we obtain
\[
   \int_{\Omega^{n+1}_{m/l}} K^+((z,w),(\xi,\eta))h(\xi,\eta)^q\dV(\xi,\eta)
   \less \left(1-\frac{\norm{z}^{2m}}{|w|^{2l}}\right)^{\gamma_1q}(1-|w|^2)^{\gamma_2q}.
\]
Since $s\le0$, $|w|^{sq}\ge1$ on $\D^*$, and therefore
\begin{equation*}
   \int_{\Omega^{n+1}_{m/l}} K^+((z,w),(\xi,\eta))h(\xi,\eta)^q\dV(\xi,\eta)
   \less h(z,w)^q.
\end{equation*}

For the second Schur estimate fix $(\xi,\eta)\in\Omega^{n+1}_{m/l}$ and integrate in
$(z,w)\in\Omega^{n+1}_{m/l}$.  Make the substitution $z=|w|^{l/m}u$.  Then
$dV(z)=|w|^{\frac{2nl}{m}}dV(u)$ and
\[
   1-\frac{\norm{z}^{2m}}{|w|^{2l}}=1-\norm{u}^{2m}.
\]
It follows that
\begin{align*}
&\int_{\Omega^{n+1}_{m/l}} K^+((z,w),(\xi,\eta))h(z,w)^p\dV(z,w) \\
&=
|\eta|^{-2\frac{nl+m-1}{m}}
\int_{\D^*}\int_{\B_n}
\frac{(1-\norm{u}^{2m})^{n+1+\gamma_1p}(1-|w|^2)^{2+\gamma_2p}|w|^{\frac{2nl}{m}+sp}}
{|1-w\overline\eta|^4
\left|1-\frac{\ip{|w|^{l/m}u}{\xi}^m}{(w\overline\eta)^l}\right|^{2(n+1)}}
\dV(u)\dV(w).
\end{align*}
The ball part now has weight $(1-\norm{u}^{2m})^{n+1+\gamma_1p}$, and
Lemma~\ref{lem:ball-estimate} gives
\[
\int_{\B_n}\frac{(1-\norm{u}^{2m})^{n+1+\gamma_1p}}
{\left|1-\frac{\ip{|w|^{l/m}u}{\xi}^m}{(w\overline\eta)^l}\right|^{2(n+1)}}\dV(u)
\less \left(1-\frac{\norm{\xi}^{2m}}{|\eta|^{2l}}\right)^{\gamma_1p},
\]
because $n+1+\gamma_1p>-1$ and $\gamma_1p<0$.  The disk part is
\[
  \int_{\D^*}
  \frac{(1-|w|^2)^{2+\gamma_2p}|w|^{\frac{2nl}{m}+sp}}
  {|1-w\overline\eta|^4}\dV(w).
\]
The choices of $\gamma_2$ and $s$ give
$2+\gamma_2p>-1$, $\gamma_2p<0$, and $\frac{2nl}{m}+sp>-2$.  Hence
\[
  \int_{\D^*}
  \frac{(1-|w|^2)^{2+\gamma_2p}|w|^{\frac{2nl}{m}+sp}}
  {|1-w\overline\eta|^4}\dV(w)
  \less (1-|\eta|^2)^{\gamma_2p}.
\]
There remains the outside factor $|\eta|^{-2(nl+m-1)/m}$ from the kernel.  Since
$s<-2(nl+m-1)/(mp)$, we have $sp<-2(nl+m-1)/m$, and hence
$|\eta|^{-2(nl+m-1)/m}\le |\eta|^{sp}$ for $0<|\eta|<1$.  Therefore
\begin{equation*}
   \int_{\Omega^{n+1}_{m/l}} K^+((z,w),(\xi,\eta))h(z,w)^p\dV(z,w)
   \less \left(1-\frac{\norm{\xi}^{2m}}{|\eta|^{2l}}\right)^{\gamma_1p}(1-|\eta|^2)^{\gamma_2p}|\eta|^{sp}
   =h(\xi,\eta)^p.
\end{equation*}
By Schur's test, the positive operator with kernel $K^+$ is bounded on
$L^p(\Omega^{n+1}_{m/l})$.  The estimate \eqref{eq:berezin-kernel-upper} then gives
the boundedness of $\calB_{m/l,n}$.
\end{proof}

\section{Necessity}\label{sec:necessity}

We now prove the converse implication.  Let $j_0\in\{0,1,\ldots,m-1\}$ be the unique integer satisfying
\begin{equation*}
   l(j_0+n)\equiv1\pmod m,
\end{equation*}
and define
\begin{equation}\label{eq:E0-beta0}
   E_0=\frac{l(j_0+n)-1}{m},
   \qquad
   \beta_0=-1-E_0.
\end{equation}
Then $E_0\in\Z_{\ge0}$.

\begin{lemma}\label{lem:critical-monomial}
Let $\alpha^0=(j_0,0,\ldots,0)\in\N^n$ and let $\beta_0$ be defined by
\eqref{eq:E0-beta0}.  Then
\[
     z^{\alpha^0}w^{\beta_0}=z_1^{j_0}w^{-1-E_0}
\]
belongs to $A^2(\Omega^{n+1}_{m/l})$.  Moreover,
\begin{equation}\label{eq:kappa-identity}
   1+E_0-\frac{lj_0}{m}=\frac{m+nl-1}{m}=\frac{nl+m-1}{m}.
\end{equation}
\end{lemma}

\begin{proof}
Using \eqref{eq:E0-beta0},
\[
  l|\alpha^0|+m\beta_0
  =lj_0-m-mE_0
  =lj_0-m-l(j_0+n)+1
  =-m-nl+1.
\]
Thus $(\alpha^0,\beta_0)$ satisfies the integer condition
\eqref{eq:integer-monomial-condition} with equality, and hence the monomial
belongs to $A^2(\Omega^{n+1}_{m/l})$.  Finally,
\[
   1+E_0-\frac{lj_0}{m}
   =1+\frac{l(j_0+n)-1}{m}-\frac{lj_0}{m}
   =\frac{m+nl-1}{m}.
\]
\end{proof}

\begin{lemma}\label{lem:adjoint-lower}
There exist a compact Reinhardt set $U\Subset\Omega^{n+1}_{m/l}$ and a constant $c>0$ such
that, for every $(\xi,\eta)\in\Omega^{n+1}_{m/l}$,
\begin{equation*}
   \calB_{m/l,n}^*\chi_U(\xi,\eta)
   \ge c\,|\xi_1|^{2j_0}|\eta|^{-2(1+E_0)}.
\end{equation*}
Here $\chi_U$ denotes the characteristic function of $U$, and $\calB_{m/l,n}^*$ denotes the Banach-space adjoint integral operator,
\begin{equation*}
   \calB_{m/l,n}^*g(\xi,\eta)
   =\int_{\Omega^{n+1}_{m/l}}
   \frac{|B_{m/l,n}((z,w),(\xi,\eta))|^2}{B_{m/l,n}((z,w),(z,w))}g(z,w)\dV(z,w).
\end{equation*}
\end{lemma}

\begin{proof}
Choose numbers $0<r_0<r_1<1$ and $\varepsilon>0$ so small that
$\varepsilon^m<r_0^l$, and set
\[
   U=\{(z,w)\in\C^n\times\C:\ \norm{z}<\varepsilon,
       \ r_0<|w|<r_1\}.
\]
Then $U\Subset\Omega^{n+1}_{m/l}$, and $U$ is a Reinhardt set.  Since $B_{m/l,n}((z,w),(z,w))$ is
continuous and positive on $U$, there is a constant $M>0$ such that
\[
   B_{m/l,n}((z,w),(z,w))\le M,\qquad (z,w)\in U.
\]
Therefore
\begin{align}\label{eq:adjoint-lower-step1}
   \calB_{m/l,n}^*\chi_U(\xi,\eta)
   &=\int_U\frac{|B_{m/l,n}((z,w),(\xi,\eta))|^2}{B_{m/l,n}((z,w),(z,w))}\dV(z,w) \notag\\
   &\ge \frac1M\int_U |B_{m/l,n}((z,w),(\xi,\eta))|^2\dV(z,w).
\end{align}

Now insert the monomial expansion \eqref{eq:full-kernel-expansion}.  For fixed
$(\xi,\eta)$, the kernel is a holomorphic function of $(z,w)$ with expansion
\[
   B_{m/l,n}((z,w),(\xi,\eta))
   =\sum_{(\alpha,\beta)\in\Lambda_{m/l,n}}
     c_{\alpha,\beta} z^\alpha w^\beta
     \overline{\xi^\alpha\eta^\beta},
   \qquad c_{\alpha,\beta}>0.
\]
Because $U$ is Reinhardt, distinct monomials are orthogonal in $L^2(U)$.  Hence
\begin{align*}
   \int_U |B_{m/l,n}((z,w),(\xi,\eta))|^2\dV(z,w)
   &=\sum_{(\alpha,\beta)\in\Lambda_{m/l,n}}
     |c_{\alpha,\beta}|^2
     \left(\int_U |z^\alpha w^\beta|^2\dV(z,w)\right)
     |\xi^\alpha|^2|\eta|^{2\beta}.
\end{align*}
All summands are nonnegative.  Keeping only the critical monomial
$(\alpha^0,
\beta_0)$ from Lemma~\ref{lem:critical-monomial}, we get
\[
   \int_U |B_{m/l,n}((z,w),(\xi,\eta))|^2\dV(z,w)
   \ge c_0 |\xi_1|^{2j_0}|\eta|^{2\beta_0}
   =c_0 |\xi_1|^{2j_0}|\eta|^{-2(1+E_0)}.
\]
Combining this with \eqref{eq:adjoint-lower-step1} proves the lemma.
\end{proof}

\begin{proposition}\label{prop:necessity}
Let $1<p<\infty$.  If $\calB_{m/l,n}$ is bounded on $L^p(\Omega^{n+1}_{m/l})$, then
$p>m+nl$.
\end{proposition}

\begin{proof}
Let $q=p/(p-1)$.  Since $\calB_{m/l,n}$ is a positive integral operator bounded
on $L^p(\Omega^{n+1}_{m/l})$, its adjoint $\calB_{m/l,n}^*$ is bounded on $L^q(\Omega^{n+1}_{m/l})$.
Choose the set $U$ from Lemma~\ref{lem:adjoint-lower}.  Since $\chi_U\in L^q(\Omega^{n+1}_{m/l})$,
we have
\[
   \calB_{m/l,n}^*\chi_U\in L^q(\Omega^{n+1}_{m/l}).
\]
By Lemma~\ref{lem:adjoint-lower}, this implies the local integrability condition
\begin{equation}\label{eq:necessary-integrability}
  \int_{0<|\eta|<\delta}
  \int_{\norm{\xi}^m<|\eta|^l}
  |\xi_1|^{2j_0q}|\eta|^{-2(1+E_0)q}
  \dV(\xi)\dV(\eta)<\infty
\end{equation}
for sufficiently small $\delta>0$.

For fixed $r=|\eta|$, the inner integral can be evaluated by the scaling
\(\xi=r^{l/m}u\), \(u\in\mathbb B_n\).  Hence
\begin{align*}
   \int_{\norm{\xi}<r^{l/m}} |\xi_1|^{2j_0q}\dV(\xi)
   &= r^{\frac{l}{m}(2j_0q+2n)}
      \int_{\mathbb B_n}|u_1|^{2j_0q}\dV(u) \\
   &= C r^{\frac{l}{m}(2j_0q+2n)}.
\end{align*}
Thus \eqref{eq:necessary-integrability} is equivalent to convergence of
\[
   \int_0^\delta
   r^{-2(1+E_0)q}
   r^{\frac{l}{m}(2j_0q+2n)}
   r\,dr.
\]
This integral converges if and only if
\[
   -2(1+E_0)q+\frac{2lj_0}{m}q+\frac{2nl}{m}>-2.
\]
Equivalently,
\[
   \left(1+E_0-\frac{lj_0}{m}\right)q<1+\frac{nl}{m}.
\]
Using \eqref{eq:kappa-identity}, we obtain
\[
   \frac{m+nl-1}{m}q<\frac{m+nl}{m},
\]
or
\begin{equation}\label{eq:q-condition}
   q<\frac{m+nl}{m+nl-1}.
\end{equation}
Since $q=p/(p-1)$, \eqref{eq:q-condition} is equivalent to
\[
   p>m+nl.
\]
This proves the proposition.
\end{proof}

\begin{proposition}\label{prop:endpoint}
The Berezin transform $\calB_{m/l,n}$ is not bounded on $L^1(\Omega^{n+1}_{m/l})$.
\end{proposition}

\begin{proof}
If $\calB_{m/l,n}$ were bounded on $L^1(\Omega^{n+1}_{m/l})$, then its Banach-space adjoint
would be bounded on $L^\infty(\Omega^{n+1}_{m/l})$.  Let $U$ be the Reinhardt set from
Lemma~\ref{lem:adjoint-lower}.  Since $\chi_U\in L^\infty(\Omega^{n+1}_{m/l})$, boundedness would
imply
\[
   \calB_{m/l,n}^*\chi_U\in L^\infty(\Omega^{n+1}_{m/l}).
\]
But Lemma~\ref{lem:adjoint-lower} gives
\[
   \calB_{m/l,n}^*\chi_U(\xi,
\eta)
   \ge c |\xi_1|^{2j_0}|\eta|^{-2(1+E_0)}.
\]
Take points with
\[
   \xi=(c_1 |\eta|^{l/m},0,
\ldots,0),
   \qquad 0<c_1<1,
\]
and let $\eta\to0$.  Then
\[
   |\xi_1|^{2j_0}|\eta|^{-2(1+E_0)}
   =c_1^{2j_0}|\eta|^{\frac{2lj_0}{m}-2(1+E_0)}.
\]
By \eqref{eq:kappa-identity},
\[
   2(1+E_0)-\frac{2lj_0}{m}=\frac{2(m+nl-1)}{m}>0.
\]
Therefore the right-hand side tends to infinity as $\eta\to0$, contradicting
$L^\infty$ boundedness of the adjoint.  Hence $\calB_{m/l,n}$ is not bounded on
$L^1(\Omega^{n+1}_{m/l})$.
\end{proof}

\begin{proof}[\textbf{Proof of  Theorem \ref{thm:main}}]
The sufficiency is Proposition~\ref{prop:sufficiency}.  The necessity for $1<p<\infty$ is
Proposition~\ref{prop:necessity}.  The endpoint $p=1$ is covered by
Proposition~\ref{prop:endpoint}.  This completes the proof.
\end{proof}
\section*{Acknowledgements}
The author would like to thank Professor Guantie Deng for his many valuable and constructive suggestions which greatly improved the presentation of the paper.

\section*{Declarations}
\textbf{Data availability.} This manuscript has no associated data.\par 
\textbf{Conflict of interest.} The author declares no competing interests.

\end{document}